\documentclass[a4paper,11pt]{article}
\usepackage[multidot]{grffile}
\usepackage{latexsym,amssymb,enumerate,amsmath,epsfig,amsthm}
\usepackage[margin=1in]{geometry}
\usepackage{setspace,color}
\usepackage{graphics}
\usepackage{graphicx}
\usepackage[ruled]{algorithm2e}
\usepackage{empheq}
\usepackage{bm,amsmath}
\usepackage{subcaption}
\usepackage{hyperref}
\hypersetup{
    colorlinks=true,
    linkcolor=blue,
    filecolor=magenta,
    urlcolor=cyan,
    citecolor=red,
    pdfpagemode=FullScreen,
}

\usepackage{rotating,booktabs,array,amsmath,amssymb}
\newcolumntype{C}{>{$\displaystyle}c<{$}}
\usepackage{multirow}
\usepackage{lineno}
\AtBeginDocument{%
  \setlength{\abovedisplayskip}{6pt}%
  \setlength{\belowdisplayskip}{6pt}%
  \setlength{\abovedisplayshortskip}{4pt}%
  \setlength{\belowdisplayshortskip}{4pt}%
}

\newcommand{\bx}{\mathbf{x}}
\newcommand{\by}{\mathbf{y}}
\newcommand{\bz}{\mathbf{z}}
\newcommand{\bp}{\mathbf{p}}

\newcommand{\ba}{\mathbf{a}}
\newcommand{\bb}{\mathbf{b}}
\newcommand{\bc}{\mathbf{c}}
\newcommand{\bu}{\mathbf{u}}
\newcommand{\bv}{\mathbf{v}}
\newcommand{\bU}{\mathbf{U}}
\newcommand{\bV}{\mathbf{V}}
\newcommand{\bA}{\mathbf{A}}
\newcommand{\bB}{\mathbf{B}}
\newcommand{\bD}{\mathbf{D}}
\newcommand{\bw}{\mathbf{w}}

\newcommand{\Log}{\operatorname{Log}}

\numberwithin{equation}{section}

\theoremstyle{plain}
\newtheorem{thm}{Theorem}[section]
\newtheorem{prop}[thm]{Proposition}
\newtheorem{lemma}[thm]{Lemma}
\newtheorem{corollary}[thm]{Corollary}

\theoremstyle{definition}

\theoremstyle{remark}
\newtheorem{remark}[thm]{Remark}

\newcommand{\Sph}{\mathbb{S}^2}

\newcommand{\slerp}{\operatorname{SLERP}}

\newcommand{\dd}{\,\mathrm{d}}

\title{Analytic First Derivatives of SIDER Interpolation}
\author{
Shingyu Leung\thanks{Department of Mathematics, the Hong Kong University of Science and Technology, Clear Water Bay, Hong Kong. Email: {\bf masyleung@ust.hk}}
}
\date{}

\begin{document}

\maketitle

\begin{abstract}
Spherical interpolation is required in numerical and geometric applications in which the unknowns are constrained to remain on the unit sphere.  Spherical Interpolation of orDER \(n\) (SIDER-\(n\)) was introduced as the high-order reconstruction component of spherical essentially non-oscillatory interpolation, where the reconstruction is built entirely from spherical linear interpolation (SLERP) operations and therefore preserves the spherical constraint exactly.  This paper develops analytic first-derivative formulas for SIDER curves of arbitrary order.  The central observation is that the recursive definition of SIDER can be differentiated by direct chain-rule propagation through its binary tree of SLERP operations.  After deriving the total derivative of SLERP with moving endpoints, we obtain compact recursions for the derivative of SIDER-\(n\), including simplified formulas at interpolation nodes and practical formulas at middle points between consecutive sampling locations.  The latter are relevant when a reconstruction is evaluated halfway between data samples, as occurs in several high-order reconstruction-based numerical algorithms.  The base case SIDER2 is treated explicitly, and SIDER3 and SIDER4 are used to illustrate the recursive mechanism.  We also prove that the derivative is tangent to the sphere at every reconstructed point, including both sampling points and middle points.  The resulting formulas extend the original SIDER/SENO framework by supplying differential information for sphere-valued reconstructions, with potential use in high-order finite-volume, ENO/WENO, SENO-type, and related methods for conservation laws and evolution problems.
\end{abstract}

\section{Introduction}

High-order interpolation and local reconstruction are central tools in numerical analysis, scientific computing, geometric modeling, and the numerical approximation of partial differential equations.  In many classical settings, the reconstructed quantity is treated as a Euclidean vector, and polynomial interpolation or polynomial reconstruction provides a natural local model.  In a growing number of applications, however, the data are constrained by geometry.  The reconstructed states may represent directions, surface normals, orientation variables, director fields, normalized magnetization vectors, or other quantities whose admissible values lie on a nonlinear manifold rather than in a linear vector space.  In such problems, the reconstruction should preserve the underlying geometric constraint throughout the parameter interval, not merely at the prescribed data points.

The model setting considered in this paper is a sequence of unit vectors
\(\bp_0,\bp_1,\ldots,\bp_n\in \Sph\subset \mathbb{R}^3 .\)
The goal is to construct a curve \(\mathbf{S}_n(t)\) that interpolates these data and remains on \(\Sph\) for all relevant parameter values.  This requirement is not automatically satisfied by standard Euclidean polynomial interpolation.  Even if all sample values have unit length, a polynomial curve passing through them generally leaves the sphere between samples.  A common remedy is to normalize the Euclidean interpolant pointwise, but this procedure changes the geometry of the construction and may complicate the analysis of accuracy, stability, and derivative information.  For applications in which the reconstructed curve is subsequently differentiated or used inside a numerical method, such a projection-based correction can also obscure the relation between the interpolation formula and the resulting tangent vectors.

A natural alternative is to build the reconstruction directly from intrinsic spherical operations.  The fundamental two-point operation is spherical linear interpolation, or SLERP.  Given two non-antipodal points on the sphere, SLERP follows the short great-circle arc joining them and therefore remains on \(\Sph\) by construction.  This operation was introduced and popularized in the computer-graphics and quaternion-interpolation literature by Shoemake \cite{shoemake_85}; further background on quaternion interpolation and related constructions may be found in \cite{dam_98}.  Since SLERP is sphere-valued, any curve obtained by composing SLERP operations is also sphere-valued, provided the relevant geodesic branches are chosen consistently.

The SIDER interpolation procedure uses precisely this principle.  It was introduced as part of the spherical essentially non-oscillatory interpolation framework, SENO, for high-order interpolation of sphere-valued data \cite{fonleu23}.  SIDER may be viewed as a spherical analogue of a de Casteljau-type construction.  In the Euclidean de Casteljau algorithm, higher-order curves are produced by repeated affine interpolation.  In SIDER, affine interpolation is replaced by SLERP.  This replacement makes the construction intrinsic to the sphere.  At the same time, SIDER differs from an ordinary B\'ezier construction because the prescribed data are interpolation points rather than merely control points.  Thus SIDER combines a recursive de Casteljau-like structure with an interpolation property adapted to spherical data.

The original SIDER/SENO framework was designed primarily for high-order, non-oscillatory reconstruction.  In that setting, SIDER supplies high-order spherical interpolants on candidate stencils, while SENO selects a locally least oscillatory candidate in the spirit of ENO methodology.  This connects SIDER/SENO to the broad family of ENO and WENO methods, which use nonlinear stencil selection or nonlinear weights to suppress oscillations near nonsmooth features \cite{harengoshcha87,shuosh88,liu_94,shu97}.  Related geometric and nonlinear interpolation ideas have appeared in spherical averaging and spline constructions \cite{busfil01}, incremental SLERP methods \cite{barhasben04}, interpolation on rotation groups and other matrix manifolds \cite{shi09,gawleo18}, and higher-order rigid-body motion interpolation \cite{haarbach_18}.  Geometric ENO-type methods have also been developed for interpolation and curve-evolution problems \cite{sidkimshu97}.

The present paper extends the SIDER framework by deriving analytic first-derivative formulas for SIDER curves of arbitrary order.  This derivative information is important for both geometric and numerical reasons.  Geometrically, the derivative of a sphere-valued reconstruction gives the tangent direction, or velocity, of the interpolating curve.  It describes how the reconstructed state changes along the sphere and provides local differential information that is not contained in the point values alone.  Numerically, derivative information is needed whenever a reconstruction is used not only to approximate values but also to approximate slopes, directional changes, or differential quantities associated with the reconstructed field.

This situation arises naturally in reconstruction-based numerical methods.  High-order finite-difference, finite-volume, ENO, and WENO schemes are built from local reconstructions whose values are evaluated at distinguished points between neighboring samples \cite{harengoshcha87,shuosh88,liu_94,shu97}.  In some algorithms, derivatives of the local reconstruction are also required.  For example, derivative information can enter generalized Riemann solvers, inverse Lax--Wendroff procedures, source-term discretizations, high-order quadrature rules, or evolution formulas that depend on local slopes \cite{liosou95,dinshuzha20}.  When such methods are applied to manifold-valued or sphere-valued data, the derivative of the reconstruction should respect the same geometry as the reconstructed values.  This is one of the motivations for developing an analytic derivative theory for SIDER.

A useful way to describe the relevant evaluation points is through the index coordinate
$
        \xi=nt$.
In this coordinate, the sampling points are located at the integers \(\xi=0,1,\ldots,n\).  Points midway between adjacent samples are located at
$
        \xi=j+\frac12$ for $j=0,\ldots,n-1 .
$
Equivalently, in the normalized SIDER parameter \(t\), these middle points correspond to
$
        t=\frac{j+1/2}{n}.
$
For example, if four data points are labeled by \(0,1,2,3\), then \(n=3\), and the middle points occur at \(\xi=0.5,1.5,2.5\), or equivalently at \(t=1/6,1/2,5/6\).  These middle-point derivatives are of independent interest in interpolation theory and become especially useful when SIDER reconstructions are incorporated into numerical schemes that evaluate reconstructed quantities between neighboring samples.

The derivative formulas developed here are not obtained by differentiating a fully expanded high-order expression.  Such an approach would quickly become unwieldy.  Instead, we use the recursive structure of SIDER itself.  The key observation is that SIDER-\(n\) is a composition of SLERP maps.  Therefore, once the total derivative of one SLERP operation with moving endpoints is known, the derivative of the entire SIDER curve follows by repeated application of the chain rule.  This gives an analytic recursion for
\(\bD_n[\bp_0,\ldots,\bp_n](t) = \frac{\dd}{\dd t}\mathbf{S}_n[\bp_0,\ldots,\bp_n](t),\)
valid for arbitrary parameter values.  The same formula can then be specialized to the sampling points and to the middle points.

An important feature of the resulting derivative is that it is tangent to the sphere without any additional projection.  Since SIDER is constructed entirely from SLERP operations, the curve satisfies
$
        \|\mathbf{S}_n(t)\|^2=1
$
for all admissible \(t\).  Differentiating this identity gives
$
        \mathbf{S}_n(t)\cdot \bD_n(t)=0,
$
so that
$
        \bD_n(t)\in T_{\mathbf{S}_n(t)}\Sph .
$
Thus, at a sampling point \(t_j=j/n\), where \(\mathbf{S}_n(t_j)=\bp_j\), the derivative lies in \(T_{\bp_j}\Sph\).  At a middle point \(t=(j+1/2)/n\), the derivative lies in the tangent space at the reconstructed middle value.  This tangent-space property is a consequence of the analytic construction itself and does not require a post-processing correction.

The main contribution of this paper is therefore a recursive, implementation-ready first-derivative theory for SIDER interpolation.  We derive the total derivative of SLERP with moving endpoints, use it to obtain the general SIDER-\(n\) derivative recursion, and then specialize the result to interpolation nodes, endpoints, and middle points.  The derivation shows that the derivative computation can be interpreted as forward-mode differentiation of the SIDER evaluation algorithm.  This viewpoint leads to compact formulas, avoids long expanded expressions for high-order curves, and is directly suitable for numerical implementation.

The paper is organized as follows.  Section~\ref{sec:background} recalls SLERP, derives the total SLERP derivative, and reviews the recursive SIDER construction.  Section~\ref{sec:general-recursion} presents the first-derivative recursion for SIDER-\(n\), including simplified formulas at interior sampling points and endpoints.  Section~\ref{sec:middle-points} derives the corresponding formulas at middle points.  Section~\ref{sec:tangent-space} proves that the derivative lies in the appropriate tangent space, both at sampling points and at middle points.  Section~\ref{sec:examples} works out SIDER2, SIDER3, and SIDER4 explicitly.  Section~\ref{sec:implementation} discusses implementation assumptions, and Section~\ref{sec:numerical-example} provides numerical visualizations and diagnostic checks.

\section{Background}\label{sec:background}

The purpose of this section is to collect the basic geometric ingredients used throughout the paper.  We first recall the SLERP formula, which provides the fundamental two-point interpolation operation on the sphere.  We then record its first-derivative formula, including the case in which both endpoints depend on the parameter.  This total derivative is the elementary building block for all subsequent SIDER derivative formulas.  Finally, we review the recursive definition of SIDER interpolation and fix the notation used later for the lower-order branches, their reparameterizations, and their derivatives.

\subsection{SLERP and its derivative}
We recall the basic SLERP identities used throughout the derivative formulas.  For non-antipodal unit vectors $\bx,\by\in\Sph$, let
\(\theta_{\bx\by}=\arccos(\bx\cdot\by)\).
The spherical linear interpolation from $\bx$ to $\by$ is
\[
        \slerp(\bx,\by;\lambda)
        =
        \frac{\sin((1-\lambda)\theta_{\bx\by})}{\sin\theta_{\bx\by}}\bx
        +
        \frac{\sin(\lambda\theta_{\bx\by})}{\sin\theta_{\bx\by}}\by .
\]
The associated spherical logarithm map is
\[
        \Log_{\bx}(\by)
        =
        \frac{\theta_{\bx\by}}{\sin\theta_{\bx\by}}
        \bigl(\by-\cos\theta_{\bx\by}\,\bx\bigr),
\]
which is a tangent vector at $\bx$ pointing along the geodesic from $\bx$ to $\by$.

The derivative formula needed for SIDER is the total derivative of SLERP with moving endpoints.  Let
\(\bz(t)=\slerp(\ba(t),\bb(t);\lambda(t)), \qquad \theta(t)=\arccos(\ba(t)\cdot\bb(t)).\)
Define
\[
        \alpha=\frac{\sin((1-\lambda)\theta)}{\sin\theta},
        \qquad
        \beta=\frac{\sin(\lambda\theta)}{\sin\theta}.
\]
Then $\bz=\alpha\ba+\beta\bb$, and differentiating gives
\begin{equation}
\label{eq:slerp-total-derivative}
\boxed{
        \mathcal{D}\slerp(\ba,\bb,\lambda;\ba',\bb',\lambda')
        =
        \alpha\ba'
        +
        \beta\bb'
        +
        (\alpha_\lambda\lambda'+\alpha_\theta\theta')\ba
        +
        (\beta_\lambda\lambda'+\beta_\theta\theta')\bb ,
}
\end{equation}
where
\(\theta' = -\frac{\ba'\cdot\bb+\ba\cdot\bb'}{\sin\theta}.\)
The scalar partial derivatives are
\[
\begin{aligned}
        \alpha_\lambda
        &=
        -\frac{\theta\cos((1-\lambda)\theta)}{\sin\theta},
        &
        \beta_\lambda
        &=
        \frac{\theta\cos(\lambda\theta)}{\sin\theta},
        \\
        \alpha_\theta
        &=
        \frac{(1-\lambda)\cos((1-\lambda)\theta)\sin\theta
        -\sin((1-\lambda)\theta)\cos\theta}{\sin^2\theta},
        &
        \beta_\theta
        &=
        \frac{\lambda\cos(\lambda\theta)\sin\theta
        -\sin(\lambda\theta)\cos\theta}{\sin^2\theta}.
\end{aligned}
\]
This identity is the differential building block used in the rest of the note.  In the special case of fixed endpoints, it reduces to
\[
        \frac{\partial}{\partial \lambda}\slerp(\bx,\by;\lambda)
        =
        \frac{\theta_{\bx\by}}{\sin\theta_{\bx\by}}
        \left[
        -\cos((1-\lambda)\theta_{\bx\by})\bx
        +
        \cos(\lambda\theta_{\bx\by})\by
        \right].
\]
In particular,
\[
        \frac{\partial}{\partial \lambda}\slerp(\bx,\by;\lambda)\bigg|_{\lambda=0}
        =
        \Log_{\bx}(\by),
        \qquad
        \frac{\partial}{\partial \lambda}\slerp(\bx,\by;\lambda)\bigg|_{\lambda=1}
        =
        -\Log_{\by}(\bx).
\]
If the two moving endpoints coincide at a parameter value, say $\ba(t_0)=\bb(t_0)$, then the smooth limiting derivative of the outer SLERP is
\[
        \frac{\dd}{\dd t}\slerp(\ba(t),\bb(t);\lambda(t))\bigg|_{t=t_0}
        =
        (1-\lambda(t_0))\ba'(t_0)+\lambda(t_0)\bb'(t_0).
\]
Changing $\lambda$ has no first-order effect in this limiting case because both endpoints of the SLERP are the same point.

\subsection{SIDER interpolation}

The SIDER construction provides an intrinsic interpolation procedure for ordered data on the sphere.  Its defining feature is that the interpolant is built entirely from spherical linear interpolation, or SLERP.  Consequently, each intermediate point generated during the construction remains on the sphere, and the final interpolating curve is sphere-valued by construction.

The idea is closely related to the classical de Casteljau algorithm for B\'ezier curves.  In the Euclidean setting, de Casteljau's algorithm repeatedly applies affine linear interpolation between control points.  SIDER follows the same recursive philosophy, but replaces affine interpolation by SLERP.  This replacement is essential: Euclidean linear interpolation between two points on the sphere generally leaves the sphere, whereas SLERP follows the short geodesic arc connecting the two points.

There is, however, an important distinction between SIDER and ordinary B\'ezier interpolation.  In a B\'ezier curve, the intermediate control points usually do not lie on the curve.  In SIDER, the prescribed data points are interpolation points.  Thus SIDER-$n$ is constructed so that, for data
\(\bp_0,\bp_1,\ldots,\bp_n\in \Sph,\)
the corresponding curve satisfies
\[
        \mathbf{S}_n[\bp_0,\ldots,\bp_n]\!\left(\frac{j}{n}\right)=\bp_j,
        \qquad j=0,\ldots,n.
\]
The parameter interval is normalized to $[0,1]$, so the data are assumed to occur at uniformly spaced parameter values.

\subsubsection{The quadratic construction: SIDER2}

The base nontrivial construction is SIDER2.  Given three points
\(\bp_0,\bp_1,\bp_2\in \Sph,\)
we first define two extrapolated spherical control points
\(\bc_\ba=\slerp(\bp_2,\bp_1;2)\) and \(\bc_\bb=\slerp(\bp_0,\bp_1;2)\).
These points are obtained by extending the geodesic arcs beyond the middle point $\bp_1$.  More precisely, $\bc_\ba$ lies on the geodesic determined by $\bp_2$ and $\bp_1$, extrapolated past $\bp_1$, while $\bc_\bb$ lies on the geodesic determined by $\bp_0$ and $\bp_1$, also extrapolated past $\bp_1$.

The SIDER2 curve is then defined by
\begin{equation}
\label{eq:sider2}
        \mathbf{S}_2[\bp_0,\bp_1,\bp_2](t)
        =
        \slerp\!\left(
        \slerp(\bp_0,\bc_\ba;t),
        \slerp(\bc_\bb,\bp_2;t);
        t
        \right),
        \qquad 0\leq t\leq 1.
\end{equation}
This formula has the same nested structure as a quadratic de Casteljau construction.  The first inner SLERP moves from $\bp_0$ to $\bc_\ba$, while the second inner SLERP moves from $\bc_\bb$ to $\bp_2$.  The outer SLERP then blends these two intermediate points using the same parameter $t$.

The extrapolated points $\bc_\ba$ and $\bc_\bb$ are chosen precisely so that the curve interpolates the middle point.  Indeed,
\(\mathbf{S}_2[\bp_0,\bp_1,\bp_2](0)=\bp_0\), \(\mathbf{S}_2[\bp_0,\bp_1,\bp_2](1)=\bp_2.\)
At the midpoint $t=1/2$, the two inner SLERP curves produce points whose geodesic midpoint is $\bp_1$.  Therefore,
\(\mathbf{S}_2[\bp_0,\bp_1,\bp_2]\!\left(\frac12\right)=\bp_1.\)
Thus SIDER2 interpolates the three prescribed data points at the normalized parameter values
\(0,\frac12,\) and 1.

\subsubsection{The recursive construction}

Higher-order SIDER curves are defined recursively.  Suppose that SIDER-$(n-1)$ has already been defined.  Given
\(\bp_0,\bp_1,\ldots,\bp_n\in \Sph,\)
we define two lower-order curves.  The left curve uses the first $n$ data points,
\(\mathbf{S}_{n-1}[\bp_0,\ldots,\bp_{n-1}],\)
while the right curve uses the last $n$ data points,
\(\mathbf{S}_{n-1}[\bp_1,\ldots,\bp_n].\)
These two curves are evaluated at shifted and rescaled parameter values.  Define
\(g_n(t)=\frac{n}{n-1}t\) and \(h_n(t)=\frac{n}{n-1}t-\frac{1}{n-1}\).
Then, for $n\geq 3$, SIDER-$n$ is defined by
\begin{equation}
\label{eq:sider-n}
        \mathbf{S}_n[\bp_0,\ldots,\bp_n](t)
        =
        \slerp\!\left(
        \mathbf{S}_{n-1}[\bp_0,\ldots,\bp_{n-1}](g_n(t)),
        \mathbf{S}_{n-1}[\bp_1,\ldots,\bp_n](h_n(t));
        t
        \right).
\end{equation}

The functions $g_n$ and $h_n$ align the lower-order curves with the interpolation nodes of the order-$n$ curve.  At the node
\(t_j=\frac{j}{n}$ for $ j=1,\ldots,n-1,\)
we have
\(g_n(t_j)=j/(n-1)\) and \(h_n(t_j)=(j-1)/(n-1)\).
Therefore,
\[
        \mathbf{S}_{n-1}[\bp_0,\ldots,\bp_{n-1}]\!\left(\frac{j}{n-1}\right)=\bp_j,
\]
and
\[
        \mathbf{S}_{n-1}[\bp_1,\ldots,\bp_n]\!\left(\frac{j-1}{n-1}\right)=\bp_j.
\]
Hence the two lower-order curves agree at each interior data point.  The outer SLERP is then a SLERP from $\bp_j$ to itself, so it also returns $\bp_j$.  This proves the interpolation property
\[
        \mathbf{S}_n[\bp_0,\ldots,\bp_n]\!\left(\frac{j}{n}\right)=\bp_j,
        \qquad j=0,\ldots,n.
\]

For example, SIDER3 is obtained from two SIDER2 curves:
\[
        \mathbf{S}_3[\bp_0,\bp_1,\bp_2,\bp_3](t)
        =
        \slerp\!\left(
        \mathbf{S}_2[\bp_0,\bp_1,\bp_2]\!\left(\frac{3t}{2}\right),
        \mathbf{S}_2[\bp_1,\bp_2,\bp_3]\!\left(\frac{3t-1}{2}\right);
        t
        \right).
\]
It interpolates the four data points at
\(0,\frac13, \frac23,1.\)
Similarly, SIDER4 is obtained from two SIDER3 curves:
\[
        \mathbf{S}_4[\bp_0,\bp_1,\bp_2,\bp_3,\bp_4](t)
        =
        \slerp\!\left(
        \mathbf{S}_3[\bp_0,\bp_1,\bp_2,\bp_3]\!\left(\frac{4t}{3}\right),
        \mathbf{S}_3[\bp_1,\bp_2,\bp_3,\bp_4]\!\left(\frac{4t-1}{3}\right);
        t
        \right),
\]
and it interpolates the five data points at
\(0, \frac14,\frac12,\frac34,1.\)

The construction is intrinsic to the sphere because every operation appearing in \eqref{eq:sider2} and \eqref{eq:sider-n} is a SLERP between unit vectors.  Thus, under the usual nondegeneracy assumptions that the relevant pairs of points are not antipodal and that a consistent geodesic branch is chosen, all intermediate quantities and all final interpolated values remain on $\Sph$.  This geometric property is the main reason SIDER is well suited for spherical interpolation problems: it preserves the manifold constraint exactly, rather than enforcing it afterward by normalization.

\section{General first-derivative recursion}\label{sec:general-recursion}

We now derive the first-derivative formula for SIDER-$n$.  The main point is that no new geometric ingredient is required beyond the derivative of a single SLERP operation.  Since SIDER-$n$ is defined recursively as a SLERP of two SIDER-$(n-1)$ curves, its derivative follows by applying the ordinary chain rule to this recursive structure.

For convenience, define
\(\bD_n[\bp_0,\ldots,\bp_n](t) := \frac{\dd}{\dd t}\mathbf{S}_n[\bp_0,\ldots,\bp_n](t).\)
Thus $\bD_n$ denotes the derivative of the SIDER-$n$ curve with respect to the normalized parameter $t\in[0,1]$.

For $n\geq 3$, recall that SIDER-$n$ is defined by
\[
        \mathbf{S}_n[\bp_0,\ldots,\bp_n](t)
        =
        \slerp\!\left(
        \mathbf{S}_{n-1}[\bp_0,\ldots,\bp_{n-1}](g_n(t)),
        \mathbf{S}_{n-1}[\bp_1,\ldots,\bp_n](h_n(t));
        t
        \right),
\]
where
\(g_n(t)=\frac{n}{n-1}t\) and \(h_n(t)=\frac{n}{n-1}t-\frac{1}{n-1}\).
Introduce the shorthand
\[
\begin{aligned}
\bA(t)&=\mathbf{S}_{n-1}[\bp_0,\ldots,\bp_{n-1}](g_n(t)),&
\bB(t)&=\mathbf{S}_{n-1}[\bp_1,\ldots,\bp_n](h_n(t)).
\end{aligned}
\]
With this notation, the recursion is simply \(\mathbf{S}_n[\bp_0,\ldots,\bp_n](t)=\slerp(\bA(t),\bB(t);t)\).

The two lower-order branches $\bA(t)$ and $\bB(t)$ themselves depend on $t$ through the affine reparameterizations $g_n$ and $h_n$.  Since
\(g_n'(t)=h_n'(t)=\frac{n}{n-1},\)
the chain rule gives
\(\bA'(t) = \frac{n}{n-1} \bD_{n-1}[\bp_0,\ldots,\bp_{n-1}](g_n(t)),\)
and
\(\bB'(t) = \frac{n}{n-1} \bD_{n-1}[\bp_1,\ldots,\bp_n](h_n(t)).\)
Therefore, differentiating the outer SLERP yields
\begin{equation}
\label{eq:general-recursion}
\boxed{
        \bD_n[\bp_0,\ldots,\bp_n](t)
        =
        \mathcal{D}\slerp
        \left(
        \bA(t),\bB(t),t;
        \bA'(t),\bB'(t),1
        \right).
}
\end{equation}
Here $\mathcal{D}\slerp(\ba,\bb,\lambda;\ba',\bb',\lambda')$ denotes the total derivative of
\(\slerp(\ba(t),\bb(t);\lambda(t))\)
when
\(\ba=\ba(t),\qquad \bb=\bb(t),\qquad \lambda=\lambda(t).\)
In the present case, the outer interpolation parameter is $\lambda(t)=t$, so $\lambda'(t)=1$.

Equation \eqref{eq:general-recursion} is the central recursive derivative formula.  It reduces the derivative of a SIDER-$n$ curve to the derivatives of two SIDER-$(n-1)$ curves.  Applying the same formula repeatedly reduces the problem to the derivative of SIDER2, which is explicit.  This is the natural chain-rule interpretation of the recursive SIDER construction.

\subsection{Interior interpolation nodes}

The formula simplifies considerably at the interpolation nodes.  Let
\(t_j=\frac{j}{n}, \qquad j=0,\ldots,n.\)
For an interior node, where $1\leq j\leq n-1$, the two lower-order branches meet at the same data point.  Indeed,
\(g_n(t_j) = \frac{n}{n-1}\frac{j}{n} = \frac{j}{n-1},\)
and
\(h_n(t_j) = \frac{n}{n-1}\frac{j}{n} - \frac{1}{n-1} = \frac{j-1}{n-1}.\)
Therefore,
\[
        \bA(t_j)
        =
        \mathbf{S}_{n-1}[\bp_0,\ldots,\bp_{n-1}]
        \!\left(\frac{j}{n-1}\right)
        =
        \bp_j,
\]
and
\[
        \bB(t_j)
        =
        \mathbf{S}_{n-1}[\bp_1,\ldots,\bp_n]
        \!\left(\frac{j-1}{n-1}\right)
        =
        \bp_j.
\]
Thus
\(\bA(t_j)=\bB(t_j)=\bp_j.\)

At such a point, the outer SLERP has equal endpoints.  The derivative of the outer SLERP then reduces to the weighted combination
\(\frac{\dd}{\dd t}\slerp(\bA(t),\bB(t);t) = (1-t)\bA'(t)+t\bB'(t),\)
evaluated at $t=t_j$.  Substituting $t_j=j/n$ and the formulas for $\bA'(t_j)$ and $\bB'(t_j)$ gives
\begin{equation}
\label{eq:interior-node}
\boxed{
        \bD_n\!\left(\frac{j}{n}\right)
        =
        \frac{n-j}{n-1}
        \bD_{n-1}[\bp_0,\ldots,\bp_{n-1}]
        \!\left(\frac{j}{n-1}\right)
        +
        \frac{j}{n-1}
        \bD_{n-1}[\bp_1,\ldots,\bp_n]
        \!\left(\frac{j-1}{n-1}\right).
}
\end{equation}
This identity is one of the most useful consequences of the recursive construction.  It shows that the derivative at an interior SIDER-$n$ node is obtained from the derivatives of the two lower-order branches that meet at that node.

One should note that the coefficients in \eqref{eq:interior-node} are not a convex pair: their sum is
\(\frac{n-j}{n-1}+\frac{j}{n-1} = \frac{n}{n-1}.\)
This factor appears because both lower-order branches are reparameterized by the scale factor $n/(n-1)$.  Thus \eqref{eq:interior-node} should be interpreted as a chain-rule weighted combination, rather than as a convex average.

\subsection{Endpoint formulas}

The endpoint formulas have a slightly different form because the two lower-order branches do not meet at the outer endpoints.

At the left endpoint $t=0$, we have
\(\bA(0) = \mathbf{S}_{n-1}[\bp_0,\ldots,\bp_{n-1}](0) = \bp_0,\)
whereas
\[
        \bB(0)
        =
        \mathbf{S}_{n-1}[\bp_1,\ldots,\bp_n]
        \!\left(-\frac{1}{n-1}\right).
\]
The right branch is therefore evaluated at a parameter value outside its interpolation interval.  This is a natural consequence of the recursive SIDER definition near the endpoint.

Since the outer SLERP parameter is $0$ at $t=0$, the endpoint derivative of the outer SLERP gives
\begin{equation}
\label{eq:left-endpoint}
\boxed{
        \bD_n(0)
        =
        \frac{n}{n-1}
        \bD_{n-1}[\bp_0,\ldots,\bp_{n-1}](0)
        +
        \Log_{\bp_0}
        \left(
        \mathbf{S}_{n-1}[\bp_1,\ldots,\bp_n]
        \!\left(-\frac{1}{n-1}\right)
        \right).
}
\end{equation}
The first term comes from the motion of the left branch.  The second term is the tangent vector at $\bp_0$ pointing toward the extrapolated value of the right branch.  It appears because the outer SLERP itself begins at $\bp_0$ and initially points toward $\bB(0)$.

Similarly, at the right endpoint $t=1$, we have
\(\bB(1) = \mathbf{S}_{n-1}[\bp_1,\ldots,\bp_n](1) = \bp_n,\)
while
\[
        \bA(1)
        =
        \mathbf{S}_{n-1}[\bp_0,\ldots,\bp_{n-1}]
        \!\left(\frac{n}{n-1}\right).
\]
Thus the left branch is evaluated beyond its interpolation interval.  Since the outer SLERP parameter is $1$ at $t=1$, the endpoint derivative gives
\begin{equation}
\label{eq:right-endpoint}
\boxed{
        \bD_n(1)
        =
        \frac{n}{n-1}
        \bD_{n-1}[\bp_1,\ldots,\bp_n](1)
        -
        \Log_{\bp_n}
        \left(
        \mathbf{S}_{n-1}[\bp_0,\ldots,\bp_{n-1}]
        \!\left(\frac{n}{n-1}\right)
        \right).
}
\end{equation}
The sign in the logarithm term is negative because, at $\lambda=1$, the outer SLERP arrives at its second endpoint $\bp_n$.  Changing the interpolation parameter moves the curve in the direction from the first endpoint toward the second endpoint, so the corresponding endpoint contribution is the negative of the logarithm from $\bp_n$ back toward the extrapolated left-branch value.

\subsection{Summary of the recursive derivative recipe}

The first derivative of SIDER-$n$ is computed as follows.

First, form the two lower-order branches
\[
        \bA(t)
        =
        \mathbf{S}_{n-1}[\bp_0,\ldots,\bp_{n-1}](g_n(t)),
        \qquad
        \bB(t)
        =
        \mathbf{S}_{n-1}[\bp_1,\ldots,\bp_n](h_n(t)).
\]
Second, compute their derivatives using
\(\bA'(t) = \frac{n}{n-1} \bD_{n-1}[\bp_0,\ldots,\bp_{n-1}](g_n(t)),\)
and
\(\bB'(t) = \frac{n}{n-1} \bD_{n-1}[\bp_1,\ldots,\bp_n](h_n(t)).\)
Third, apply the total SLERP derivative:
\[
        \bD_n[\bp_0,\ldots,\bp_n](t)
        =
        \mathcal{D}\slerp
        \left(
        \bA(t),\bB(t),t;
        \bA'(t),\bB'(t),1
        \right).
\]

At interior interpolation nodes, this reduces to \eqref{eq:interior-node}.  At the two endpoints, it reduces to \eqref{eq:left-endpoint} and \eqref{eq:right-endpoint}.  These formulas provide a complete recursive procedure for differentiating SIDER curves of arbitrary order.

\section{Derivatives at middle points}\label{sec:middle-points}

The preceding recursion gives \(\bD_n(t)\) at any admissible parameter value.  We now record a specialization that is useful when the reconstruction is evaluated halfway between consecutive sampling locations.  It is convenient to introduce the index coordinate
\(\xi=nt,\qquad 0\leq \xi\leq n,\)
so that the sampling points are located at \(\xi=0,1,\ldots,n\).  The middle point between the samples \(\bp_j\) and \(\bp_{j+1}\) is then
\(\xi_{j+1/2}=j+\frac12,\qquad t_{j+1/2}=\frac{j+1/2}{n},\qquad j=0,\ldots,n-1.\)
If \(\widehat{\mathbf{S}}_n(\xi)=\mathbf{S}_n(\xi/n)\), then the derivative with respect to the index coordinate is
\(\widehat{\bD}_n(\xi) := \frac{\dd}{\dd \xi}\widehat{\mathbf{S}}_n(\xi) = \frac1n\,\bD_n(\xi/n).\)
Thus the formulas below give derivatives with respect to the normalized parameter \(t\); derivatives with respect to \(\xi\) are obtained by multiplying by \(1/n\).

For \(n\geq3\), define the left and right branches as in Section~\ref{sec:general-recursion}.  At the middle point \(t_{j+1/2}\), their parameters are
\(g_n(t_{j+1/2})=\frac{j+1/2}{n-1},\qquad h_n(t_{j+1/2})=\frac{j-1/2}{n-1}.\)
Set
\[
\begin{aligned}
        \bA_{j+1/2}
        &=
        \mathbf{S}_{n-1}[\bp_0,\ldots,\bp_{n-1}]
        \!\left(\frac{j+1/2}{n-1}\right),\\
        \bB_{j+1/2}
        &=
        \mathbf{S}_{n-1}[\bp_1,\ldots,\bp_n]
        \!\left(\frac{j-1/2}{n-1}\right),
\end{aligned}
\]
and
\[
\begin{aligned}
        \bU_{j+1/2}
        &=
        \frac{n}{n-1}
        \bD_{n-1}[\bp_0,\ldots,\bp_{n-1}]
        \!\left(\frac{j+1/2}{n-1}\right),\\
        \bV_{j+1/2}
        &=
        \frac{n}{n-1}
        \bD_{n-1}[\bp_1,\ldots,\bp_n]
        \!\left(\frac{j-1/2}{n-1}\right).
\end{aligned}
\]
Then the first derivative at the middle point is
\begin{equation}
\label{eq:middle-point-general}
\boxed{
        \bD_n\!\left(t_{j+1/2}\right)
        =
        \mathcal{D}\slerp
        \left(
        \bA_{j+1/2},\bB_{j+1/2},t_{j+1/2};
        \bU_{j+1/2},\bV_{j+1/2},1
        \right).
}
\end{equation}
Unlike the derivative formula at interpolation nodes, there is generally no equal-endpoint simplification here.  The two lower-order branch values \(\bA_{j+1/2}\) and \(\bB_{j+1/2}\) are usually distinct, and the full total SLERP derivative is needed.  Near the two outer ends, one of the lower-order branches may be evaluated outside \([0,1]\), just as in the endpoint formulas.

For the base case \(n=2\), the two middle points occur at \(\xi=1/2\) and \(\xi=3/2\), corresponding to \(t=1/4\) and \(t=3/4\).  With the notation of the SIDER2 construction,
\[
        \bD_2\!\left(\frac14\right)
        =
        \mathcal{D}\slerp
        \left(
        \bA_2\!\left(\frac14\right),\bB_2\!\left(\frac14\right),\frac14;
        \bA_2'\!\left(\frac14\right),\bB_2'\!\left(\frac14\right),1
        \right),
\]
and
\[
        \bD_2\!\left(\frac34\right)
        =
        \mathcal{D}\slerp
        \left(
        \bA_2\!\left(\frac34\right),\bB_2\!\left(\frac34\right),\frac34;
        \bA_2'\!\left(\frac34\right),\bB_2'\!\left(\frac34\right),1
        \right).
\]

For SIDER3, with four data points labelled \(0,1,2,3\), the middle locations are \(\xi=0.5,1.5,2.5\), or equivalently \(t=1/6,1/2,5/6\).  Let \(S_2^L=S_2[\bp_0,\bp_1,\bp_2]\) and \(S_2^R=S_2[\bp_1,\bp_2,\bp_3]\).  The derivatives at these three middle points are
\[
\begin{aligned}
        \bD_3\!\left(\frac16\right)
        &=
        \mathcal{D}\slerp\left(
        S_2^L\!\left(\frac14\right),S_2^R\!\left(-\frac14\right),\frac16;
        \frac32D_2^L\!\left(\frac14\right),\frac32D_2^R\!\left(-\frac14\right),1\right),\\
        \bD_3\!\left(\frac12\right)
        &=
        \mathcal{D}\slerp\left(
        S_2^L\!\left(\frac34\right),S_2^R\!\left(\frac14\right),\frac12;
        \frac32D_2^L\!\left(\frac34\right),\frac32D_2^R\!\left(\frac14\right),1\right),\\
        \bD_3\!\left(\frac56\right)
        &=
        \mathcal{D}\slerp\left(
        S_2^L\!\left(\frac54\right),S_2^R\!\left(\frac34\right),\frac56;
        \frac32D_2^L\!\left(\frac54\right),\frac32D_2^R\!\left(\frac34\right),1\right).
\end{aligned}
\]
Here \(D_2^L\) and \(D_2^R\) denote the derivatives of the left and right SIDER2 branches.  If the derivative is desired with respect to the index coordinate \(\xi\), each of these expressions is multiplied by \(1/3\).

For SIDER4, the same rule gives the middle locations \(\xi=0.5,1.5,2.5,3.5\), equivalently \(t=1/8,3/8,5/8,7/8\).  Let \(S_3^L=S_3[\bp_0,\bp_1,\bp_2,\bp_3]\) and \(S_3^R=S_3[\bp_1,\bp_2,\bp_3,\bp_4]\).  Then
\[
\begin{aligned}
        \bD_4\!\left(\frac18\right)
        &=
        \mathcal{D}\slerp\left(
        S_3^L\!\left(\frac16\right),S_3^R\!\left(-\frac16\right),\frac18;
        \frac43D_3^L\!\left(\frac16\right),\frac43D_3^R\!\left(-\frac16\right),1\right),\\
        \bD_4\!\left(\frac38\right)
        &=
        \mathcal{D}\slerp\left(
        S_3^L\!\left(\frac12\right),S_3^R\!\left(\frac16\right),\frac38;
        \frac43D_3^L\!\left(\frac12\right),\frac43D_3^R\!\left(\frac16\right),1\right),\\
        \bD_4\!\left(\frac58\right)
        &=
        \mathcal{D}\slerp\left(
        S_3^L\!\left(\frac56\right),S_3^R\!\left(\frac12\right),\frac58;
        \frac43D_3^L\!\left(\frac56\right),\frac43D_3^R\!\left(\frac12\right),1\right),\\
        \bD_4\!\left(\frac78\right)
        &=
        \mathcal{D}\slerp\left(
        S_3^L\!\left(\frac76\right),S_3^R\!\left(\frac56\right),\frac78;
        \frac43D_3^L\!\left(\frac76\right),\frac43D_3^R\!\left(\frac56\right),1\right).
\end{aligned}
\]
Again, each \(D_3\) term may be expanded using the SIDER3 recursion, and each \(D_2\) term ultimately reduces to the explicit SIDER2 derivative.  These formulas show that middle-point derivatives require no new theory; they are obtained by evaluating the same recursive derivative at half-integer positions in the index coordinate.

\section{Tangent-space character of the derivative formula}\label{sec:tangent-space}

The derivative formulas derived above are written as ambient vectors in \(\mathbb{R}^3\).  Since the reconstructed values lie on \(\mathbb{S}^2\), the desired derivative should belong to the tangent plane of the sphere at the reconstructed point.  This statement should not be interpreted merely as the elementary fact that the derivative of an exact sphere-valued curve is tangent to the sphere.  The more relevant point for the present construction is that the explicit algebraic formula used to compute the derivative has the tangent-space property by itself.  In other words, the recursive formula does not have to be followed by a tangent-plane projection.

We first prove this property for one SLERP operation.  The result is stated at the level of the derivative operator \(\mathcal{D}\slerp\), rather than at the level of an abstract differentiable curve.

\begin{lemma}[Tangency preservation of the SLERP derivative formula]\label{lem:slerp-formula-tangent}
Let \(\ba,\bb\in\mathbb{S}^2\) be non-antipodal, and let \(\bu\in T_{\ba}\mathbb{S}^2\), \(\bv\in T_{\bb}\mathbb{S}^2\).  Let \(\mu\in\mathbb{R}\) be the rate of change of the interpolation parameter.  Define
\[
        \theta=\arccos(\ba\cdot\bb),\qquad
        \alpha=\frac{\sin((1-\lambda)\theta)}{\sin\theta},
        \qquad
        \beta=\frac{\sin(\lambda\theta)}{\sin\theta},
\]
and set
\[
        \bz=\slerp(\ba,\bb;\lambda)=\alpha\ba+\beta\bb .
\]
Let
\[
        \dot\theta
        =
        -\frac{\bu\cdot\bb+\ba\cdot\bv}{\sin\theta},
\]
and define
\[
        \dot\alpha=\alpha_\lambda \mu+\alpha_\theta\dot\theta,
        \qquad
        \dot\beta=\beta_\lambda \mu+\beta_\theta\dot\theta .
\]
The algebraic SLERP derivative formula
\[
        \bw
        =
        \mathcal{D}\slerp(\ba,\bb,\lambda;\bu,\bv,\mu)
        =
        \alpha\bu+\beta\bv+\dot\alpha\,\ba+\dot\beta\,\bb
\]
satisfies
\[
        \bz\cdot\bw=0 .
\]
Consequently,
\[
        \mathcal{D}\slerp(\ba,\bb,\lambda;\bu,\bv,\mu)
        \in T_{\slerp(\ba,\bb;\lambda)}\mathbb{S}^2 .
\]
\end{lemma}

\begin{proof}
The proof is an algebraic verification of the formula.  Since \(\bu\in T_{\ba}\mathbb{S}^2\) and \(\bv\in T_{\bb}\mathbb{S}^2\), we have
\[
        \ba\cdot\bu=0,\qquad \bb\cdot\bv=0 .
\]
Moreover, \(\ba\cdot\bb=\cos\theta\), and the prescribed value of \(\dot\theta\) is exactly the algebraic derivative of the relation \(\cos\theta=\ba\cdot\bb\) under the perturbations \(\bu\) and \(\bv\).  Indeed,
\[
        -\sin\theta\,\dot\theta
        =
        \bu\cdot\bb+\ba\cdot\bv .
\]

Now compute the normal component of \(\bw\) at \(\bz\).  Since \(\bz=\alpha\ba+\beta\bb\), we obtain
\[
\begin{aligned}
        \bz\cdot\bw
        &=
        (\alpha\ba+\beta\bb)\cdot
        \left(\alpha\bu+\beta\bv+\dot\alpha\,\ba+\dot\beta\,\bb\right)  \\
        &=
        \alpha\beta(\ba\cdot\bv+\bb\cdot\bu)
        +
        \dot\alpha(\alpha+\beta\cos\theta)
        +
        \dot\beta(\alpha\cos\theta+\beta).
\end{aligned}
\]
Using \(\ba\cdot\bv+\bb\cdot\bu=-\sin\theta\,\dot\theta\), this becomes
\[
        \bz\cdot\bw
        =
        -\alpha\beta\sin\theta\,\dot\theta
        +
        \dot\alpha(\alpha+\beta\cos\theta)
        +
        \dot\beta(\alpha\cos\theta+\beta).
\]
It remains to show that the right-hand side is zero.

The SLERP coefficients satisfy the identity
\[
        \alpha^2+\beta^2+2\alpha\beta\cos\theta=1 .
\]
This identity follows directly from the definitions of \(\alpha\) and \(\beta\), using
\[
        \sin^2 A+\sin^2 B+2\sin A\sin B\cos(A+B)=\sin^2(A+B),
\]
with \(A=(1-\lambda)\theta\) and \(B=\lambda\theta\).  Differentiating the identity
\[
        \alpha^2+\beta^2+2\alpha\beta\cos\theta=1
\]
with respect to the independent variables \(\lambda\) and \(\theta\), and then applying the increments \(\mu\) and \(\dot\theta\), gives
\[
        \dot\alpha(\alpha+\beta\cos\theta)
        +
        \dot\beta(\alpha\cos\theta+\beta)
        -
        \alpha\beta\sin\theta\,\dot\theta
        =
        0 .
\]
Therefore \(\bz\cdot\bw=0\).  Hence the vector produced by the explicit formula \(\mathcal{D}\slerp\) is tangent to the sphere at the SLERP output \(\bz\).
\end{proof}

We now apply this algebraic invariance recursively to SIDER.  The point of the next proposition is that tangency is preserved by the derivative computation itself.

\begin{prop}[Tangency preservation of the recursive SIDER derivative formula]\label{prop:sider-formula-tangent}
Assume that all SLERP operations appearing in the SIDER construction are nonsingular.  Let
\[
        \mathbf{S}_n(t)=\mathbf{S}_n[\bp_0,\ldots,\bp_n](t)
\]
be the SIDER-\(n\) reconstruction, and let
\[
        \bD_n(t)=\bD_n[\bp_0,\ldots,\bp_n](t)
\]
be the vector produced by the recursive derivative formula \eqref{eq:general-recursion}, with the SIDER2 derivative used as the base case.  Then, in exact arithmetic,
\[
        \mathbf{S}_n(t)\cdot\bD_n(t)=0
\]
for every admissible \(t\).  Equivalently,
\[
        \bD_n(t)\in T_{\mathbf{S}_n(t)}\mathbb{S}^2 .
\]
\end{prop}

\begin{proof}
The proof is by induction on the order \(n\), but the induction is applied to the derivative formula rather than to an already-known smooth sphere-valued curve.

For the base case \(n=2\), the SIDER2 construction is
\[
        \mathbf{S}_2(t)=\slerp(\mathbf{A}_2(t),\mathbf{B}_2(t);t),
\]
where \(\mathbf{A}_2(t)\) and \(\mathbf{B}_2(t)\) are themselves SLERP curves.  The derivatives \(\mathbf{A}_2'(t)\) and \(\mathbf{B}_2'(t)\) are obtained from the fixed-endpoint SLERP derivative formula.  By Lemma~\ref{lem:slerp-formula-tangent}, these derivative vectors satisfy
\[
        \mathbf{A}_2(t)\cdot\mathbf{A}_2'(t)=0,
        \qquad
        \mathbf{B}_2(t)\cdot\mathbf{B}_2'(t)=0 .
\]
Therefore the two branch derivative vectors supplied to the outer SLERP derivative formula are tangent to their respective branch values.  Applying Lemma~\ref{lem:slerp-formula-tangent} once more to the outer SLERP gives
\[
        \mathbf{S}_2(t)\cdot\bD_2(t)=0 .
\]
Thus the derivative vector produced by the SIDER2 formula lies in \(T_{\mathbf{S}_2(t)}\mathbb{S}^2\).

Assume now that the statement holds for SIDER-\((n-1)\).  For SIDER-\(n\), write
\[
        \mathbf{S}_n(t)=\slerp(\mathbf{A}(t),\mathbf{B}(t);t),
\]
where
\[
        \mathbf{A}(t)
        =
        \mathbf{S}_{n-1}[\bp_0,\ldots,\bp_{n-1}](g_n(t)),
        \qquad
        \mathbf{B}(t)
        =
        \mathbf{S}_{n-1}[\bp_1,\ldots,\bp_n](h_n(t)).
\]
The recursive derivative formula uses
\[
        \mathbf{A}'_{\rm alg}(t)
        =
        \frac{n}{n-1}
        \bD_{n-1}[\bp_0,\ldots,\bp_{n-1}](g_n(t)),
\]
and
\[
        \mathbf{B}'_{\rm alg}(t)
        =
        \frac{n}{n-1}
        \bD_{n-1}[\bp_1,\ldots,\bp_n](h_n(t)).
\]
By the induction hypothesis,
\[
        \mathbf{A}(t)\cdot \mathbf{A}'_{\rm alg}(t)=0,
        \qquad
        \mathbf{B}(t)\cdot \mathbf{B}'_{\rm alg}(t)=0,
\]
because multiplication by the scalar factor \(n/(n-1)\) does not change tangency.  Thus the inputs passed to the outer SLERP derivative formula satisfy precisely the hypotheses of Lemma~\ref{lem:slerp-formula-tangent}.  Therefore
\[
        \mathbf{S}_n(t)\cdot
        \mathcal{D}\slerp
        \left(
        \mathbf{A}(t),\mathbf{B}(t),t;
        \mathbf{A}'_{\rm alg}(t),\mathbf{B}'_{\rm alg}(t),1
        \right)
        =
        0 .
\]
By definition, the vector on the second factor is exactly the recursively computed derivative \(\bD_n(t)\).  Hence
\[
        \mathbf{S}_n(t)\cdot\bD_n(t)=0 .
\]
This completes the induction.
\end{proof}

\begin{corollary}[Tangency at interpolation nodes]\label{cor:sampling-point-tangent}
At each interpolation node \(t_j=j/n\), \(j=0,\ldots,n\), the derivative vector computed by the recursive formula satisfies
\[
        \bD_n(t_j)\in T_{\bp_j}\mathbb{S}^2,
        \qquad
        \bp_j\cdot\bD_n(t_j)=0 .
\]
\end{corollary}

\begin{proof}
The SIDER interpolation property gives \(\mathbf{S}_n(t_j)=\bp_j\).  Evaluating Proposition~\ref{prop:sider-formula-tangent} at \(t=t_j\) gives
\[
        \bp_j\cdot\bD_n(t_j)
        =
        \mathbf{S}_n(t_j)\cdot\bD_n(t_j)
        =
        0 .
\]
Therefore \(\bD_n(t_j)\in T_{\bp_j}\mathbb{S}^2\).
\end{proof}

\begin{corollary}[Tangency at middle points]\label{cor:middle-point-tangent}
Let
\[
        t_{j+1/2}=\frac{j+1/2}{n},
        \qquad
        j=0,\ldots,n-1,
\]
and define the reconstructed middle value by
\[
        \mathbf{m}_{j+1/2}
        =
        \mathbf{S}_n[\bp_0,\ldots,\bp_n](t_{j+1/2}) .
\]
Then the derivative vector computed by the recursive formula satisfies
\[
        \bD_n(t_{j+1/2})\in T_{\mathbf{m}_{j+1/2}}\mathbb{S}^2,
        \qquad
        \mathbf{m}_{j+1/2}\cdot\bD_n(t_{j+1/2})=0 .
\]
If the derivative is expressed in the index coordinate \(\xi=nt\), namely
\[
        \widehat{\bD}_n(j+1/2)
        =
        \frac{1}{n}\bD_n(t_{j+1/2}),
\]
then
\[
        \widehat{\bD}_n(j+1/2)\in T_{\mathbf{m}_{j+1/2}}\mathbb{S}^2
\]
as well.
\end{corollary}

\begin{proof}
By Proposition~\ref{prop:sider-formula-tangent},
\[
        \mathbf{S}_n(t)\cdot\bD_n(t)=0
\]
for every admissible \(t\).  Taking \(t=t_{j+1/2}\) gives
\[
        \mathbf{m}_{j+1/2}\cdot\bD_n(t_{j+1/2})=0 .
\]
This is exactly the condition that
\[
        \bD_n(t_{j+1/2})\in T_{\mathbf{m}_{j+1/2}}\mathbb{S}^2 .
\]
The index-coordinate derivative differs only by multiplication by the scalar \(1/n\), and hence belongs to the same tangent space.
\end{proof}

\begin{remark}[No projection is hidden in the formula]
The above result is a property of the complete recursive derivative expression.  Individual terms in \(\mathcal{D}\slerp\), such as \(\alpha\bu\), \(\beta\bv\), \(\dot\alpha\,\ba\), and \(\dot\beta\,\bb\), need not separately be tangent to the final point \(\bz=\slerp(\ba,\bb;\lambda)\).  Their normal components cancel exactly in the full expression because of the algebraic identities satisfied by the SLERP weights.  Thus the recursive derivative formula is intrinsically tangent in exact arithmetic.  In floating-point implementation, the residual
\[
        \left|\mathbf{S}_n(t)\cdot\bD_n(t)\right|
\]
should therefore be at the level of roundoff error, apart from possible loss of accuracy near singular or nearly singular SLERP configurations.  Any additional tangent-plane projection should be viewed only as a numerical stabilization or diagnostic step, not as part of the analytic formula.
\end{remark}

\section{Examples}\label{sec:examples}

The general recursion derived above gives a compact formula for the first derivative of SIDER curves of arbitrary order.  We now specialize the construction to SIDER2, SIDER3, and SIDER4 in order to make the chain-rule structure explicit.  The SIDER2 case is treated in detail because it is the base case for all higher-order formulas.  The SIDER3 case is the first genuinely recursive example and shows how the lower-order derivatives enter the outer SLERP derivative.  The SIDER4 case is presented more concisely: it records the recursive setup and representative formulas, while avoiding a repetitive list of all node and middle-point derivatives.

\subsection{SIDER2}

We begin with the base case SIDER2, since all higher-order derivative formulas eventually reduce to this construction.  Given \(\bp_0,\bp_1,\bp_2\in\Sph\), define \(\bc_\ba=\slerp(\bp_2,\bp_1;2)\) and \(\bc_\bb=\slerp(\bp_0,\bp_1;2)\).  The SIDER2 curve is obtained by two inner SLERP curves followed by one outer SLERP.  Set \(\bA_2(t)=\slerp(\bp_0,\bc_\ba;t)\) and \(\bB_2(t)=\slerp(\bc_\bb,\bp_2;t)\); equivalently,
\[
\mathbf{S}_2[\bp_0,\bp_1,\bp_2](t)=\slerp(\bA_2(t),\bB_2(t);t).
\]
The two inner curves $\bA_2$ and $\bB_2$ play the role of the two first-level curves in a quadratic de Casteljau construction.  The first starts at $\bp_0$ and moves toward the extrapolated point $\bc_\ba$, while the second starts at $\bc_\bb$ and moves toward $\bp_2$.  The outer SLERP then blends these two moving points using the same parameter $t$.

Since the endpoints of $\bA_2$ and $\bB_2$ are fixed, their derivatives are ordinary fixed-endpoint SLERP derivatives:
\(\bA_2'(t) = \frac{\partial}{\partial \lambda} \slerp(\bp_0,\bc_\ba;\lambda)\bigg|_{\lambda=t},\)
and
\(\bB_2'(t) = \frac{\partial}{\partial \lambda} \slerp(\bc_\bb,\bp_2;\lambda)\bigg|_{\lambda=t}.\)
For fixed endpoints $\bx,\by\in \Sph$, with
\(\theta_{xy}=\arccos(\bx\cdot \by),\)
we have
\[
        \frac{\partial}{\partial \lambda}\slerp(\bx,\by;\lambda)
        =
        \frac{\theta_{xy}}{\sin\theta_{xy}}
        \left[
        -\cos((1-\lambda)\theta_{xy})\bx
        +
        \cos(\lambda\theta_{xy})\by
        \right].
\]
Therefore,
\[
        \bA_2'(t)
        =
        \frac{\theta_{0a}}{\sin\theta_{0a}}
        \left[
        -\cos((1-t)\theta_{0a})\bp_0
        +
        \cos(t\theta_{0a})\bc_\ba
        \right],
\]
where
\(\theta_{0a}=\arccos(\bp_0\cdot \bc_\ba),\)
and
\[
        \bB_2'(t)
        =
        \frac{\theta_{b2}}{\sin\theta_{b2}}
        \left[
        -\cos((1-t)\theta_{b2})\bc_\bb
        +
        \cos(t\theta_{b2})\bp_2
        \right],
\]
where
\(\theta_{b2}=\arccos(\bc_\bb\cdot \bp_2).\)

The derivative of SIDER2 is now obtained by applying the total SLERP derivative to the outer interpolation:
\begin{equation}
\label{eq:d2-general}
\boxed{
        \bD_2[\bp_0,\bp_1,\bp_2](t)
        =
        \mathcal{D}\slerp
        \left(
        \bA_2(t),\bB_2(t),t;
        \bA_2'(t),\bB_2'(t),1
        \right).
}
\end{equation}
This formula is the SIDER2 version of the general chain-rule recipe.  It is already completely explicit, because $\bA_2(t)$, $\bB_2(t)$, $\bA_2'(t)$, and $\bB_2'(t)$ are all given in closed form, and the operator $\mathcal{D}\slerp$ was written explicitly in \eqref{eq:slerp-total-derivative}.

\subsubsection{Endpoint derivatives}

The endpoint derivatives have particularly simple forms.  At \(t=0\), we have \(\bA_2(0)=\bp_0\) and \(\bB_2(0)=\bc_\bb\).
The motion of the left branch contributes
\(\bA_2'(0)=\Log_{\bp_0}(\bc_\ba),\)
while the variation of the outer SLERP parameter contributes the initial geodesic velocity from $\bp_0$ toward $\bc_\bb$, namely
\(\Log_{\bp_0}(\bc_\bb).\)
Thus
\[
\boxed{
        \bD_2[\bp_0,\bp_1,\bp_2](0)
        =
        \Log_{\bp_0}(\bc_\ba)
        +
        \Log_{\bp_0}(\bc_\bb).
}
\]

At $t=1$, we have
\(\bA_2(1)=\bc_\ba, \qquad \bB_2(1)=\bp_2.\)
The motion of the right branch gives
\(\bB_2'(1)=-\Log_{\bp_2}(\bc_\bb),\)
and the endpoint contribution from the outer SLERP is
\(-\Log_{\bp_2}(\bc_\ba).\)
Therefore,
\[
\boxed{
        \bD_2[\bp_0,\bp_1,\bp_2](1)
        =
        -\Log_{\bp_2}(\bc_\ba)
        -
        \Log_{\bp_2}(\bc_\bb).
}
\]

These endpoint formulas have a natural geometric interpretation.  At the left endpoint, the derivative is the sum of two tangent directions based at $\bp_0$: one coming from the motion of the inner curve $\bA_2$, and one coming from the outer SLERP toward $\bB_2(0)=\bc_\bb$.  At the right endpoint, the same mechanism occurs in reverse, producing two negative logarithmic terms based at $\bp_2$.

\subsubsection{Derivative at the middle interpolation point}

The middle interpolation point occurs at
\(t=\frac12,\)
where
\(\mathbf{S}_2[\bp_0,\bp_1,\bp_2]\!\left(\frac12\right)=\bp_1.\)
However, the two inner points
\[
        \bA_2\!\left(\frac12\right),
        \qquad
        \bB_2\!\left(\frac12\right)
\]
need not be equal.  Rather, $\bp_1$ is obtained as the geodesic midpoint between them under the outer SLERP.  Thus the equal-endpoint simplification used later for interior nodes of SIDER-$n$, $n\geq 3$, does not apply directly to SIDER2.

The derivative at the midpoint is therefore
\[
\boxed{
        \bD_2[\bp_0,\bp_1,\bp_2]\!\left(\frac12\right)
        =
        \mathcal{D}\slerp
        \left(
        \bA_2\!\left(\frac12\right),
        \bB_2\!\left(\frac12\right),
        \frac12;
        \bA_2'\!\left(\frac12\right),
        \bB_2'\!\left(\frac12\right),
        1
        \right).
}
\]
To make this expression more explicit, define
\[
        \ba=\bA_2\!\left(\frac12\right),
        \qquad
        \bb=\bB_2\!\left(\frac12\right),
\]
and
\[
        \bu=\bA_2'\!\left(\frac12\right),
        \qquad
        \bv=\bB_2'\!\left(\frac12\right).
\]
Let
\(\phi=\arccos(\ba\cdot \bb).\)
Since
\(\bp_1=\slerp\!\left(\ba,\bb;\frac12\right),\)
we have
\(\bp_1= \frac{\ba+\bb}{2\cos(\phi/2)}.\)
The derivative of the angle $\phi$ induced by the moving endpoints $\ba$ and $\bb$ is
\(\phi' = -\frac{\bu\cdot \bb+\ba\cdot \bv}{\sin\phi}.\)
Using the total SLERP derivative at $\lambda=1/2$ gives
\[
\boxed{
\begin{aligned}
        \bD_2[\bp_0,\bp_1,\bp_2]\!\left(\frac12\right)
        &=
        \frac{1}{2\cos(\phi/2)}(\bu+\bv)
        +
        \frac{\tan(\phi/2)}{4\cos(\phi/2)}\phi'(\ba+\bb)
        \\
        &\qquad
        +
        \frac{\phi}{2\sin(\phi/2)}(\bb-\ba).
\end{aligned}
}
\]
Here the first term comes from the motion of the two inner branches, the second term accounts for the variation of the spherical angle between them, and the third term comes from the change of the outer SLERP parameter itself.

This completes the base case.  The formulas above provide both the general derivative $\bD_2(t)$ and the special values at the three interpolation nodes
\(0,\qquad \frac12,\qquad 1.\)
All higher-order SIDER derivative formulas reduce recursively to these SIDER2 expressions.

\subsection{SIDER3}

We next apply the recursive derivative formula to SIDER3.  This is the first case in which the general SIDER recursion appears explicitly.  Given four points
\(\bp_0,\bp_1,\bp_2,\bp_3\in \Sph,\)
define the two SIDER2 branches
\[
        \mathbf{S}_2^L(t)
        =
        \mathbf{S}_2[\bp_0,\bp_1,\bp_2](t),
        \qquad
        \mathbf{S}_2^R(t)
        =
        \mathbf{S}_2[\bp_1,\bp_2,\bp_3](t).
\]
The left branch interpolates
\(\bp_0,\quad \bp_1,\quad \bp_2,\)
while the right branch interpolates
\(\bp_1,\quad \bp_2,\quad \bp_3.\)
The SIDER3 curve is obtained by blending these two lower-order curves:
\[
        \mathbf{S}_3[\bp_0,\bp_1,\bp_2,\bp_3](t)
        =
        \slerp\left(
        \mathbf{S}_2^L\!\left(\frac{3t}{2}\right),
        \mathbf{S}_2^R\!\left(\frac{3t-1}{2}\right);
        t
        \right).
\]
Thus SIDER3 has the same structure as the general recursion with
\(g_3(t)=\frac{3t}{2}, \qquad h_3(t)=\frac{3t-1}{2}.\)
It interpolates the four data points at
\(t=0,\qquad t=\frac13,\qquad t=\frac23,\qquad t=1.\)

For notational clarity, define
\[
        \bA_3(t)
        =
        \mathbf{S}_2^L\!\left(\frac{3t}{2}\right),
        \qquad
        \bB_3(t)
        =
        \mathbf{S}_2^R\!\left(\frac{3t-1}{2}\right).
\]
Then
\(\mathbf{S}_3[\bp_0,\bp_1,\bp_2,\bp_3](t) = \slerp(\bA_3(t),\bB_3(t);t).\)
Differentiating the two branches gives
\(\bA_3'(t) = \frac32 \bD_2^L\!\left(\frac{3t}{2}\right),\)
and
\(\bB_3'(t) = \frac32 \bD_2^R\!\left(\frac{3t-1}{2}\right),\)
where
\[
        \bD_2^L(t)
        =
        \frac{\dd}{\dd t}\mathbf{S}_2^L(t),
        \qquad
        \bD_2^R(t)
        =
        \frac{\dd}{\dd t}\mathbf{S}_2^R(t).
\]
Therefore, for arbitrary $t$,
\begin{equation}
\label{eq:d3-general}
\boxed{
        \bD_3[\bp_0,\bp_1,\bp_2,\bp_3](t)
        =
        \mathcal{D}\slerp
        \left(
        \bA_3(t),\bB_3(t),t;
        \bA_3'(t),\bB_3'(t),1
        \right).
}
\end{equation}
This formula is fully explicit once the SIDER2 derivative formulas from the previous section are substituted.

\subsubsection{Derivative at the left endpoint}

At the left endpoint $t=0$, we have
\(\bA_3(0) = \mathbf{S}_2^L(0) = \bp_0,\)
whereas
\(\bB_3(0) = \mathbf{S}_2^R\!\left(-\frac12\right).\)
Thus the right SIDER2 branch is evaluated outside its interpolation interval.  This extrapolated value determines the initial direction of the outer SLERP.

Using the endpoint formula \eqref{eq:left-endpoint}, or equivalently differentiating the outer SLERP at parameter $0$, gives
\[
\boxed{
        \bD_3[\bp_0,\bp_1,\bp_2,\bp_3](0)
        =
        \frac32\bD_2^L(0)
        +
        \Log_{\bp_0}\!\left(
        \mathbf{S}_2^R\!\left(-\frac12\right)
        \right).
}
\]
The first term is the derivative of the left SIDER2 branch, multiplied by the reparameterization factor $3/2$.  The second term is the initial geodesic velocity of the outer SLERP from $\bp_0$ toward the extrapolated right-branch value.

\subsubsection{Derivative at the first interior node}

The first interior node is
\(t=\frac13,\)
where the SIDER3 curve interpolates $\bp_1$.  At this parameter value,
\[
        \bA_3\!\left(\frac13\right)
        =
        \mathbf{S}_2^L\!\left(\frac12\right)
        =
        \bp_1,
\]
and
\(\bB_3\!\left(\frac13\right) = \mathbf{S}_2^R(0) = \bp_1.\)
Thus the two lower-order branches meet at the same point.  The equal-endpoint SLERP derivative therefore applies.  Since the outer SLERP parameter is $t=1/3$, we obtain
\[
\begin{aligned}
        \bD_3\!\left(\frac13\right)
        &=
        \left(1-\frac13\right)\bA_3'\!\left(\frac13\right)
        +
        \frac13 \bB_3'\!\left(\frac13\right)
        \\
        &=
        \frac23
        \left[
        \frac32\bD_2^L\!\left(\frac12\right)
        \right]
        +
        \frac13
        \left[
        \frac32\bD_2^R(0)
        \right].
\end{aligned}
\]
Hence
\[
\boxed{
        \bD_3[\bp_0,\bp_1,\bp_2,\bp_3]\!\left(\frac13\right)
        =
        \bD_2^L\!\left(\frac12\right)
        +
        \frac12\bD_2^R(0).
}
\]

\subsubsection{Derivative at the second interior node}

The second interior node is
\(t=\frac23,\)
where the SIDER3 curve interpolates $\bp_2$.  At this node,
\(\bA_3\!\left(\frac23\right) = \mathbf{S}_2^L(1) = \bp_2,\)
and
\[
        \bB_3\!\left(\frac23\right)
        =
        \mathbf{S}_2^R\!\left(\frac12\right)
        =
        \bp_2.
\]
Again the two branches meet, so the equal-endpoint SLERP derivative gives
\[
\begin{aligned}
        \bD_3\!\left(\frac23\right)
        &=
        \left(1-\frac23\right)\bA_3'\!\left(\frac23\right)
        +
        \frac23 \bB_3'\!\left(\frac23\right)
        \\
        &=
        \frac13
        \left[
        \frac32\bD_2^L(1)
        \right]
        +
        \frac23
        \left[
        \frac32\bD_2^R\!\left(\frac12\right)
        \right].
\end{aligned}
\]
Therefore
\[
\boxed{
        \bD_3[\bp_0,\bp_1,\bp_2,\bp_3]\!\left(\frac23\right)
        =
        \frac12\bD_2^L(1)
        +
        \bD_2^R\!\left(\frac12\right).
}
\]

\subsubsection{Derivative at the right endpoint}

At the right endpoint $t=1$, the right branch reaches the final data point:
\(\bB_3(1) = \mathbf{S}_2^R(1) = \bp_3.\)
The left branch, however, is evaluated beyond its interpolation interval:
\(\bA_3(1) = \mathbf{S}_2^L\!\left(\frac32\right).\)
Using the endpoint formula \eqref{eq:right-endpoint}, we obtain
\[
\boxed{
        \bD_3[\bp_0,\bp_1,\bp_2,\bp_3](1)
        =
        \frac32\bD_2^R(1)
        -
        \Log_{\bp_3}\!\left(
        \mathbf{S}_2^L\!\left(\frac32\right)
        \right).
}
\]
The first term is the derivative of the right SIDER2 branch, again multiplied by the reparameterization factor $3/2$.  The second term is the endpoint contribution from the outer SLERP.

\subsubsection{Summary for SIDER3}

Combining the four node formulas, we have
\[
\boxed{
        \bD_3(0)
        =
        \frac32\bD_2^L(0)
        +
        \Log_{\bp_0}\!\left(\mathbf{S}_2^R\!\left(-\frac12\right)\right),
}
\]
\[
\boxed{
        \bD_3\!\left(\frac13\right)
        =
        \bD_2^L\!\left(\frac12\right)
        +
        \frac12\bD_2^R(0),
}
\]
\[
\boxed{
        \bD_3\!\left(\frac23\right)
        =
        \frac12\bD_2^L(1)
        +
        \bD_2^R\!\left(\frac12\right),
}
\]
and
\[
\boxed{
        \bD_3(1)
        =
        \frac32\bD_2^R(1)
        -
        \Log_{\bp_3}\!\left(\mathbf{S}_2^L\!\left(\frac32\right)\right).
}
\]

The two interior formulas illustrate the main simplification that occurs at interpolation nodes.  At $\bp_1$ and $\bp_2$, the two lower-order SIDER2 branches meet exactly, so the derivative of the outer SLERP reduces to a chain-rule weighted combination of the two branch derivatives.  The coefficients are not a convex pair; their sum is $3/2$, reflecting the fact that both SIDER2 branches are reparameterized by the factor $3/2$ inside the SIDER3 construction.

\subsection{SIDER4}

The SIDER4 case adds one further level to the same recursion.  It is useful to record its structure, but it is not necessary to list every endpoint, node, and middle-point formula separately, since all of them follow immediately from the general results above.  Given five points \(\bp_0,\ldots,\bp_4\in\Sph\), set
\[
        \mathbf{S}_3^L(t)=\mathbf{S}_3[\bp_0,\bp_1,\bp_2,\bp_3](t),
        \qquad
        \mathbf{S}_3^R(t)=\mathbf{S}_3[\bp_1,\bp_2,\bp_3,\bp_4](t).
\]
Then
\[
        \mathbf{S}_4(t)
        =
        \slerp\left(
        \mathbf{S}_3^L\!\left(\frac{4t}{3}\right),
        \mathbf{S}_3^R\!\left(\frac{4t}{3}-\frac13\right);
        t
        \right),
\]
so \(\mathbf{S}_4\) interpolates \(\bp_0,\ldots,\bp_4\) at \(t=0,1/4,1/2,3/4,1\).

Define
\[
        \bA_4(t)=\mathbf{S}_3^L\!\left(\frac{4t}{3}\right),
        \qquad
        \bB_4(t)=\mathbf{S}_3^R\!\left(\frac{4t}{3}-\frac13\right).
\]
Then \(\mathbf{S}_4(t)=\slerp(\bA_4(t),\bB_4(t);t)\), with
\[
        \bA_4'(t)=\frac43\bD_3^L\!\left(\frac{4t}{3}\right),
        \qquad
        \bB_4'(t)=\frac43\bD_3^R\!\left(\frac{4t}{3}-\frac13\right),
\]
where \(\bD_3^L=(\mathbf{S}_3^L)'\) and \(\bD_3^R=(\mathbf{S}_3^R)'\).  Therefore,
\[
\boxed{
        \bD_4(t)
        =
        \mathcal{D}\slerp
        \left(
        \bA_4(t),\bB_4(t),t;
        \bA_4'(t),\bB_4'(t),1
        \right).
}
\]
This is the complete SIDER4 derivative formula in recursive form.  It differs from the SIDER3 formula only in that the lower-order branches are SIDER3 curves rather than SIDER2 curves.

As a representative interpolation-node formula, consider the middle node \(t=1/2\), where both branches meet at \(\bp_2\):
\[
        \bA_4\!\left(\frac12\right)=\mathbf{S}_3^L\!\left(\frac23\right)=\bp_2,
        \qquad
        \bB_4\!\left(\frac12\right)=\mathbf{S}_3^R\!\left(\frac13\right)=\bp_2.
\]
Using the equal-endpoint SLERP derivative gives
\[
\boxed{
        \bD_4\!\left(\frac12\right)
        =
        \frac23\bD_3^L\!\left(\frac23\right)
        +
        \frac23\bD_3^R\!\left(\frac13\right).
}
\]
The coefficients sum to \(4/3\), not to \(1\), because both SIDER3 branches are reparameterized by the factor \(4/3\).

The same recursive expression also gives the derivatives at the middle points.  In the index coordinate \(\xi=4t\), these points are \(\xi=j+1/2\), \(j=0,\ldots,3\), corresponding to \(t=1/8,3/8,5/8,7/8\).  For example, at \(t=3/8\), which lies between the samples \(\bp_1\) and \(\bp_2\), we have
\[
        \bA_4\!\left(\frac38\right)=\mathbf{S}_3^L\!\left(\frac12\right),
        \qquad
        \bB_4\!\left(\frac38\right)=\mathbf{S}_3^R\!\left(\frac16\right),
\]
and hence
\[
\boxed{
        \bD_4\!\left(\frac38\right)
        =
        \mathcal{D}\slerp
        \left(
        \bA_4\!\left(\frac38\right),
        \bB_4\!\left(\frac38\right),
        \frac38;
        \frac43\bD_3^L\!\left(\frac12\right),
        \frac43\bD_3^R\!\left(\frac16\right),
        1
        \right).
}
\]
The other SIDER4 sampling-point and middle-point derivatives follow analogously.  Expanding any \(\bD_3\) term reduces the expression to the SIDER3 formulas above, and then to the SIDER2 base case.  This illustrates the main practical advantage of the recursive formulation: higher-order derivatives can be evaluated exactly without writing increasingly long order-specific expressions.

\section{Implementation remarks and assumptions}\label{sec:implementation}

We conclude with several implementation remarks and assumptions that are important for using the preceding derivative formulas in practice.

First, all formulas assume that the SLERP operations are well defined.  If
\(\theta=\arccos(\bx\cdot \by),\)
then the standard SLERP formula contains the factor $\sin\theta$ in the denominator.  Therefore, numerical difficulties occur when $\theta$ is close to either $0$ or $\pi$.  The case $\theta\approx 0$ is usually harmless and may be treated by replacing SLERP with its limiting linearized form.  The case $\theta\approx \pi$ is more serious: antipodal or nearly antipodal points do not determine a unique short geodesic on the sphere.  Hence the SIDER construction requires a consistent branch choice and, in practice, assumes that the relevant point pairs are sufficiently far from antipodal configurations.

Second, one should keep track of the geodesic branch used by each SLERP.  The SIDER construction is local in spirit: it is intended for ordered data whose adjacent points are close enough that the desired spherical arc is unambiguous.  This assumption is especially important when using extrapolated values such as
\(\slerp(\bp_2,\bp_1;2), \qquad \slerp(\bp_0,\bp_1;2),\)
or higher-order extrapolations generated recursively.  These extrapolated values are still defined by the same SLERP formula, but they may leave the short arc between the original points.  The derivative formulas remain valid provided the chosen branch is used consistently.

Third, the recursive formulas naturally evaluate lower-order SIDER curves outside the interval $[0,1]$ at the outer endpoints.  For example, SIDER3 involves the values
\[
        \mathbf{S}_2^R\!\left(-\frac12\right)
        \qquad\text{and}\qquad
        \mathbf{S}_2^L\!\left(\frac32\right),
\]
while SIDER4 involves
\[
        \mathbf{S}_3^R\!\left(-\frac13\right)
        \qquad\text{and}\qquad
        \mathbf{S}_3^L\!\left(\frac43\right).
\]
These terms should be interpreted as spherical extrapolations.  They are not additional assumptions or separate definitions; they are produced by evaluating the same recursive SIDER formula at parameter values outside the interpolation interval.  The endpoint derivative formulas are valid as long as these extrapolated quantities are nonsingular.

Fourth, the derivative $\bD_n(t)$ is computed as a vector in the ambient space $\mathbb{R}^3$.  Since the SIDER curve itself remains on the unit sphere, this derivative must be tangent to the sphere at the point $\mathbf{S}_n(t)$.  In exact arithmetic,
\(\|\mathbf{S}_n(t)\|^2=1\)
implies
\(\mathbf{S}_n(t)\cdot \bD_n(t)=0.\)
Thus
\(\bD_n(t)\in T_{\mathbf{S}_n(t)}\Sph.\)
In floating-point computations, the scalar quantity \(|\mathbf{S}_n(t)\cdot\bD_n(t)|\) is useful as a diagnostic for roundoff error.  For a correctly implemented analytic derivative it should be close to machine precision, because tangency follows from the identity \(\|\mathbf{S}_n(t)\|=1\).

Fifth, when the data are parameterized over a physical interval rather than over the normalized interval $[0,1]$, a scaling factor must be applied.  The formulas in this note differentiate with respect to the normalized parameter $t$.  If the physical parameter is $T\in[T_0,T_n]$ and
\(t=\frac{T-T_0}{T_n-T_0},\)
then
\(\frac{\dd}{\dd T} = \frac{1}{T_n-T_0} \frac{\dd}{\dd t}.\)
Therefore, the derivative with respect to physical time is
\(\frac{\dd}{\dd T}\mathbf{S}_n(T) = \frac{1}{T_n-T_0}\bD_n(t).\)
For nonuniform data spacing, the recursive construction and the derivative formulas would need to be modified accordingly.

Finally, the recursive derivative formula may be viewed as a manual version of forward-mode automatic differentiation.  A routine that evaluates SIDER-$n$ can be adapted to return both the point value and its derivative.  At each recursive node, the routine evaluates the two lower-order branches, computes their derivatives, and then applies the total SLERP derivative
\[
        \mathcal{D}\slerp
        \left(
        \bA(t),\bB(t),t;
        \bA'(t),\bB'(t),1
        \right).
\]
This avoids expanding high-order SIDER derivatives into long closed-form expressions.  The implementation remains compact, follows the recursive definition directly, and preserves the analytic structure of the derivative computation.

\section{Numerical example}\label{sec:numerical-example}

We now give a numerical example to illustrate the derivative formulas.  The purpose of this example is not to test convergence, but to verify the geometric properties predicted by the analysis: the SIDER curve should interpolate the prescribed data points, and the computed first derivative should be tangent to the sphere at each interpolation node.

We use five points on the unit sphere generated by
\(u_i=\frac{i}{4},\qquad i=0,\ldots,4,\)
\(\ell_i=0.20+1.50\pi u_i, \qquad \varphi_i=0.45\sin(2\pi u_i)+0.15u_i,\)
and
\[
        \bp_i=
        \begin{bmatrix}
        \cos\varphi_i\cos\ell_i\\
        \cos\varphi_i\sin\ell_i\\
        \sin\varphi_i
        \end{bmatrix}.
\]
The same data are then used to construct three examples: SIDER2 with \(\bp_0,\bp_1,\bp_2\), SIDER3 with \(\bp_0,\bp_1,\bp_2,\bp_3\), and SIDER4 with \(\bp_0,\bp_1,\bp_2,\bp_3,\bp_4\).
For each case, the SIDER curve is evaluated on a fine grid in the normalized parameter interval $[0,1]$.  The first derivative is computed using the recursive chain-rule formula derived above.  At the interpolation nodes, tangent vectors are drawn at the corresponding data points.
\begin{figure}[!htbp]
        \centering
        (a) \includegraphics[width=0.95\textwidth]{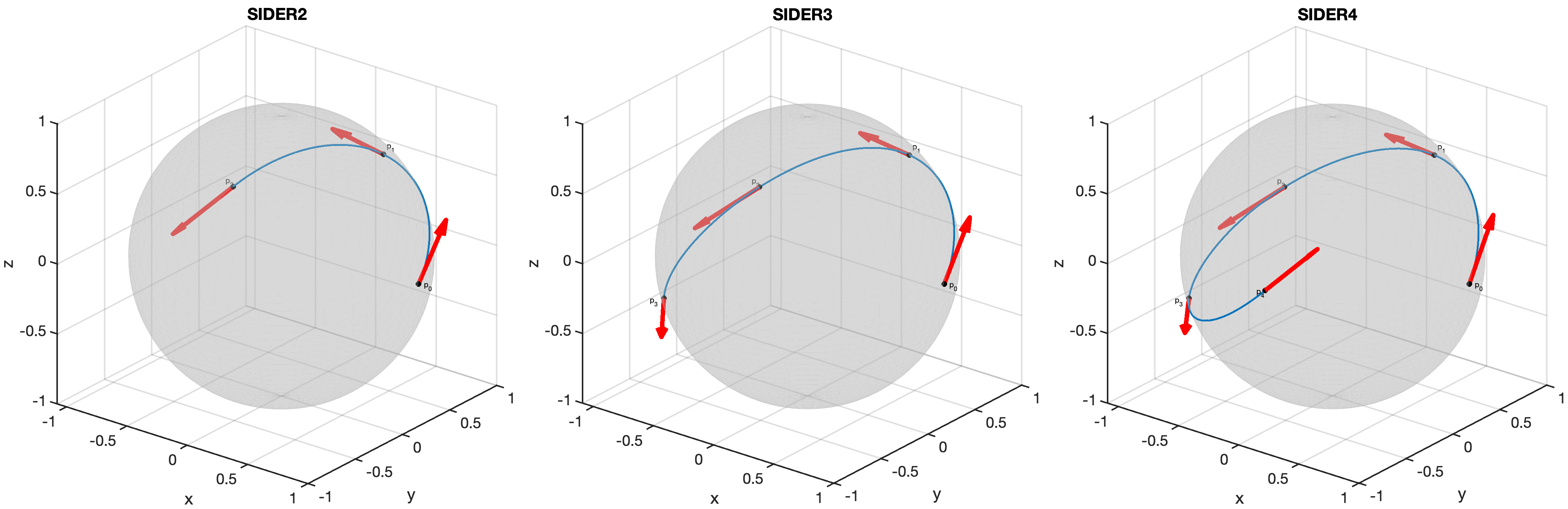} \\
        (b) \includegraphics[width=0.95\textwidth]{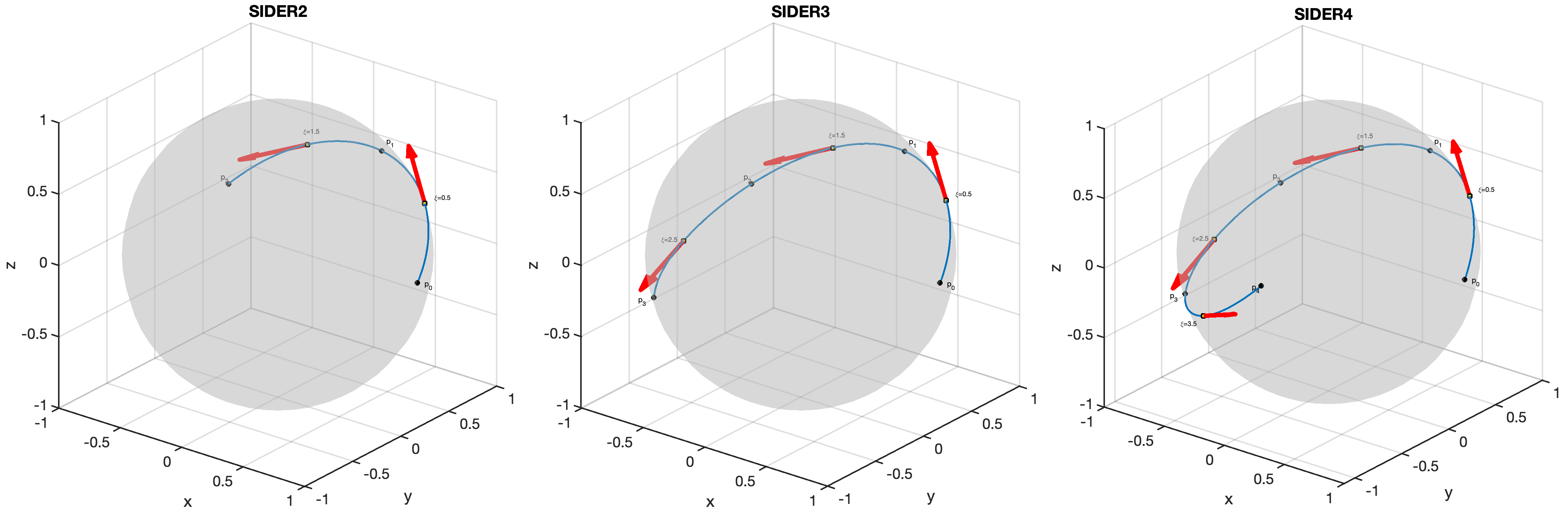}
        \caption{
        Numerical visualization of SIDER interpolants and their first derivatives.
        (a) shows derivatives at interpolation nodes, while (b) shows derivatives at middle points.
        }
        \label{fig:sider-derivative-demo}
\end{figure}

Figure~\ref{fig:sider-derivative-demo} illustrates the recursive derivative construction for SIDER2, SIDER3, and SIDER4.  In both panels, the solid curves are the SIDER interpolants on the unit sphere and the black markers are the prescribed data points.  The arrows represent the first derivatives computed from the analytic recursive formula.  They are normalized only for visualization; the actual derivative magnitudes are those obtained from the formula.

Figure~\ref{fig:sider-derivative-demo}(a) shows the derivatives at the interpolation nodes.  At these points, the curve satisfies
\(\mathbf{S}_n\!\left(\frac{j}{n}\right)=\mathbf{p}_j,\)
and the derivative belongs to the tangent space at the corresponding data point:
\[
        \mathbf{D}_n\!\left(\frac{j}{n}\right)
        \in T_{\mathbf{p}_j}\mathbb{S}^2.
\]
Figure~\ref{fig:sider-derivative-demo}(b) shows the derivatives at the middle points, defined in the index coordinate \(\xi=nt\) by
\(\xi=j+\frac12.\)
Equivalently, these points correspond to
\(t=\frac{j+1/2}{n}.\)
At each such location, the derivative is tangent to the sphere at the reconstructed middle point
\[
        \mathbf{m}_{j+1/2}
        =
        \mathbf{S}_n\!\left(\frac{j+1/2}{n}\right).
\]
Thus
\[
        \mathbf{D}_n\!\left(\frac{j+1/2}{n}\right)
        \in T_{\mathbf{m}_{j+1/2}}\mathbb{S}^2.
\]

The figure confirms the geometric properties established analytically.  Since the SIDER curve is constructed entirely from SLERP operations, it remains on the unit sphere.  Differentiating the identity
\(\|\mathbf{S}_n(t)\|^2=1\)
gives
\(\mathbf{S}_n(t)\cdot \mathbf{D}_n(t)=0,\)
which explains why all displayed derivative arrows lie tangent to the sphere.  The same implementation evaluates SIDER2, SIDER3, and SIDER4 without requiring separately expanded formulas for each order.

\section{Conclusion}

We have developed an analytic first-derivative theory for SIDER interpolation on the unit sphere.  The main result is a recursive chain-rule formula for \(\bD_n(t)=\frac{\dd}{\dd t}\mathbf{S}_n(t)\), obtained by differentiating the SLERP tree that defines SIDER-\(n\).  The only elementary building block required is the total derivative of a single SLERP operation with moving endpoints and a moving interpolation parameter.  Once this block is available, the derivative of any higher-order SIDER curve follows by differentiating the two lower-order branches and applying the same SLERP derivative at the outer node.

The resulting formulas have several useful forms.  For arbitrary parameter values, the derivative is computed by a forward recursive evaluation that returns both the SIDER value and its derivative.  At interpolation nodes, the formula simplifies because the two lower-order branches meet at the same data point.  The derivative at an interior node is then a chain-rule weighted combination of two lower-order derivatives, while the endpoint derivatives contain logarithmic terms associated with the extrapolated branch values generated by the SIDER recursion.  We also derived corresponding formulas at middle points between consecutive sampling locations.  These middle-point derivatives do not usually admit the equal-endpoint simplification, but they are obtained directly from the same total SLERP derivative and therefore require no additional analytic machinery.

We also proved that the derivative produced by this analytic construction is tangent to the sphere at the reconstructed point.  Thus \(\bD_n(t)\in T_{\mathbf{S}_n(t)}\Sph\) for every admissible parameter value.  In particular, at a sampling node \(t_j=j/n\) one has \(\bD_n(t_j)\in T_{\bp_j}\Sph\), and at a middle point \(t_{j+1/2}=(j+1/2)/n\) one has \(\bD_n(t_{j+1/2})\in T_{\mathbf{S}_n(t_{j+1/2})}\Sph\).  This property is a direct consequence of the fact that SIDER is composed entirely of SLERP operations and therefore remains on \(\Sph\).  No separate tangent-plane correction is part of the mathematical formulation.

The derivative recursion extends the earlier SIDER/SENO interpolation framework \cite{fonleu23} by providing differential information for the reconstructed spherical curve.  This is useful for visualization and geometric analysis, but it is also relevant for numerical methods in which local reconstructions must be differentiated.  In particular, high-order finite-volume, ENO/WENO, SENO, and related reconstruction-based methods often require reconstructed values and, in some formulations, derivatives at distinguished evaluation points between neighboring samples.  The formulas derived here provide such derivatives for sphere-valued reconstructions while preserving the underlying manifold constraint.

Future work may include incorporating these derivative formulas into high-order finite-volume or semi-Lagrangian schemes for \(\Sph\)-valued fields, extending the construction to nonuniform parameter spacing, and developing analogous derivative recursions for related interpolation procedures on other manifolds such as rotation groups and more general embedded submanifolds.

\section*{Computational Methodology and Disclosure}

This research utilized GPT-5.5 (Thinking) to perform primary computational modeling and mathematical derivation. The author acted as the principal investigator, defining the research parameters and theoretical scope. The author notes that while the computational workflow was executed via AI, the author maintains full responsibility for the study's conclusions. The author acknowledges that the mathematical depth of the generated results exceeds current manual verification capabilities and presents these findings as AI-assisted hypotheses subject to future formal peer verification. Consequently, this article is intended solely for dissemination as a preprint on arXiv and is not submitted to peer-reviewed journals in its current form.

\bibliographystyle{plain}
\bibliography{syleung}

\end{document}